\documentclass[12pt,a4paper]{amsart}
\usepackage[margin=1.35in]{geometry}
\usepackage{amscd,amsmath,amsxtra,amsthm,amssymb,stmaryrd,xr,mathrsfs,mathtools,enumerate,commath, comment}
\usepackage{stmaryrd}
\usepackage{xcolor}
\usepackage{commath}
\usepackage{comment}
\usepackage{tikz-cd}
\usepackage{longtable} 
\usepackage{pdflscape} 
\usepackage{booktabs}
\usepackage{hyperref}
\definecolor{vegasgold}{rgb}{0.77, 0.7, 0.35}
\definecolor{darkgoldenrod}{rgb}{0.72, 0.53, 0.04}
\definecolor{gold(metallic)}{rgb}{0.83, 0.69, 0.22}
\hypersetup{
 colorlinks=true,
 linkcolor=darkgoldenrod,
 filecolor=brown,      
 urlcolor=gold(metallic),
 citecolor=darkgoldenrod,
 pdftitle={Catalan's equation over function fields},
 }

\usepackage[all,cmtip]{xy}

\DeclareFontFamily{U}{wncy}{}
\DeclareFontShape{U}{wncy}{m}{n}{<->wncyr10}{}
\DeclareSymbolFont{mcy}{U}{wncy}{m}{n}
\DeclareMathSymbol{\Sh}{\mathord}{mcy}{"58}
\usepackage[T2A,T1]{fontenc}
\usepackage[OT2,T1]{fontenc}

\newtheorem{theorem}{Theorem}[section]
\newtheorem{lemma}[theorem]{Lemma}

\newtheorem*{theorem*}{Theorem}
\newtheorem*{corollary*}{Corollary}

\numberwithin{equation}{section}

\theoremstyle{remark}
\newtheorem{remark}[theorem]{Remark}

\newcommand{\Z}{\mathbb{Z}}

\newcommand{\Q}{\mathbb{Q}}
\newcommand{\F}{\mathbb{F}}

\newcommand{\cO}{\mathcal{O}}

\newcommand{\op}[1]{\operatorname{#1}}

\begin{document}
\title[Catalan's equation over function fields]{Remarks on Catalan's equation over function fields}

\author[A.~Ray]{Anwesh Ray}
\address[A.~Ray]{Centre de recherches mathématiques,
Université de Montréal,
Pavillon André-Aisenstadt,
2920 Chemin de la tour,
Montréal (Québec) H3T 1J4, Canada}
\email{anwesh.ray@umontreal.ca}

\begin{abstract}
 Let $\ell$ be a prime number, $F$ be a global function field of characteristic $\ell$. Assume that there is a prime $P_\infty$ of degree $1$. Let $\cO_F$ be the ring of functions in $F$ with no poles outside of $\{P_\infty\}$. We study solutions to Catalan's equation $X^m-Y^n=1$ over $\cO_F$ and show that under certain additional conditions, there are no non-constant solutions which lie in $\cO_F$, when $m,n>1$.
\end{abstract}


\maketitle
\section{Introduction}
Let $m>1$ and $n>1$ be integers, and consider the diophantine equation
\[X^m-Y^n=1.\]The famous Catalan conjecture states that there are no non-trivial integer solutions to the above equation except when $m=2$, $n=3$ and $(X,Y)=(\pm 3, 2)$. The celebrated result of Mih$\breve{\text{a}}$ilescu resolves this conjecture using techniques from the theory of cyclotomic fields (cf. \cite{mihailescu2004primary}). Given the close analogy between number fields and function fields, it is of interest to study analogues of Catalan's conjecture in characteristic $\ell>0$. The field of rational numbers $\Q$ is the simplest number field to consider, and analogously, the most natural analogue is the field of rational functions $\F(T)$, where $T$ is a formal variable, and $\F$ is a finite field. The ring of integers $\Z$ is thus analogous to the ring of polynomial functions $\F[T]$, which shares similar properties to $\Z$. The reader is referred to \cite{rosen2002number,goldschmidt2006algebraic} for an introduction to the arithmetic of function fields, and further perspectives elaborating the close analogy between number fields and their counterparts in positive characteristic. 

\par Let $\ell$ be a prime number and $F$ be a global function field of characteristic $\ell$. Denote by $\F_\ell$ the finite field with $\ell$ elements and set $\kappa$ to denote the algebraic closure of $\F_\ell$ in $F$. Note that $\kappa$ is a finite field (by assumption). Recall (from \cite[Chapter 5]{rosen2002number}) that a \emph{prime} in $F$ is defined to be the maximal ideal $v$ of a discrete valuation ring $R$ contained in $F$, with fraction field equal to $F$. The \emph{degree} of $v$ is defined to be the dimension of $R/v$ over the field of constants $\kappa$. Each prime $v$ comes equipped with a valuation $\op{ord}_v:F \rightarrow \Z\cup \{\infty\}$. Assume that there exists a prime $P_\infty$ of $F$ which has degree $1$, and let $\cO_F$ be the ring of functions in $F$ with no poles outside $\{P_\infty\}$. The point $P_\infty$ is referred to as the \emph{point at infinity} and $\cO_F$ is the \emph{ring of integers} of $F$. We say that a solution $(X,Y)\in \cO_F^2$ to $X^m-Y^n=1$ is \emph{constant} if $X$ and $Y$ are both contained in $\kappa$, and \emph{non-constant} otherwise.
\par Recall from \emph{loc. cit.} that a divisor is a finite integral linear combination of primes of $F$. The principal divisor associated to $g\in F$ is denoted $\op{div}(g)$, and two divisors $D_1$ and $D_2$ are said to be equivalent if $D_1-D_2$ is a principal divisor. The group of divisors classes of degree $0$ is finite (cf. \cite[Lemma 5.6]{rosen2002number}), and its cardinality is the \emph{class number of $F$}, and this quantity is denoted by $h_F$. Given a prime number $p\neq \ell$, let $F(\mu_p)$ be the function field obtained by adjoining the $p$-th roots of unity $\mu_p$ to $F$. Note that $F(\mu_p)=\kappa'\cdot F$, where $\kappa'=\kappa(\mu_p)$. Thus, $F(\mu_p)$ is a constant field extension of $F$ in the sense of \cite[Chapter 8]{rosen2002number}.
\begin{theorem}\label{main}
    Let $F$ be a global function field of characteristic $\ell>0$. Let $p$ and $q$ be prime numbers and assume that all the following conditions are satisfied
    \begin{enumerate}
        \item\label{c1 of main} $p\neq \ell$ and $q\neq \ell$, 
        \item\label{c2 of main} if $p\neq q$, then either $q\nmid h_{F(\mu_p)}$ or $p\nmid h_{F(\mu_q)}$.
        \item\label{c3 of main} if $q=2$, $p\neq 2$ and $q\mid h_{F(\mu_p)}$, then $p\nmid h_{F(\mu_4)}$.
    \end{enumerate} Then, there are no non-constant solutions to $X^p-Y^q=1$ in $\cO_F$. More generally, if $m>1$ and $n>1$ are integers such that $m$ is divisible by a prime $p$ and $n$ by a prime $q$ for which the above conditions are satisfied, then there are no non-constant solutions to $X^m-Y^n=1$ in $\cO_F$.
\end{theorem}
 The condition requiring that $p$ and $q$ are distinct from $\ell$ is necessary, since if $m=\ell$ for instance, it is easy to construct a large class of non-constant solutions if one of the primes is equal to $\ell$ (cf. Remark \ref{remark} for details). 

\par We mention some related work of relevance. Silverman \cite{silverman1982catalan} considered a general class of equations of the form $a X^m+bY^n=c$ over a general function field $K$, and proved that under some further conditions, there are only finitely many solutions when $a,b,c\in K^*$ are fixed. There is a mistake in the statement of Silverman's result, which has been corrected by Koymans \cite{koymans2022generalized}. The result of Koymans moreover applies to fields of larger dimension. The Catalan equation was studied by Nathanson \cite{nathanson1974catalan} over $K[T]$ and $K(T)$ where $K$ is a field of positive characteristic. It is shown in \emph{loc. cit.} that if $m>1$ and $n>1$ are coprime to $\ell$ then there are no solutions to Catalan's equation $X^m-Y^n=1$ that lie in $K[T]$ but not in $K$. Specializing to the case when $K$ is a finite field, one obtains the conclusion of Theorem \ref{main} for the rational function field. This is because the class number of any rational function field is equal to $0$. Theorem \ref{main} can thus be viewed as a generalization of Nathanson's result to general function fields $F$ with added stipulations on $(m,n)$.

\subsection{Acknowledgment:} The author thanks Peter Koymans for a helpful suggestion. The author is supported by the CRM-Simons postdoctoral fellowship. He thanks the referees for their helpful reports.

\section{Proof of the main result}
\par Recall that $F$ is a global function field of characteristic $\ell>0$ with field of constants $\kappa$. Let $\bar{\kappa}$ be the algebraic closure of $\kappa$ in a fixed algebraic closure of $F$, and set $F'$ to denote the composite $F\cdot \bar{\kappa}$. Also, denote by $A$ the composite $\cO_F\cdot \bar{\kappa}$. The field $F'$ is identified with the function field of a projective curve $\mathfrak{X}$ over $\bar{\kappa}$ and each point in $\mathfrak{X}(\bar{\kappa})$ corresponds to a valuation ring $R\subset F'$ with residue field $\bar{\kappa}$ and fraction field $F'$. The valuation ring associated to $w\in \mathfrak{X}(\bar{\kappa})$ is denoted $\cO_w$, and we refer to $w$ as a \emph{prime of $F'$}. We say that $w$ divides (or lies above) a prime $v$ of $F$ if there is a natural inclusion of valuation rings $\cO_v\hookrightarrow \cO_w$ induced by the inclusion $F\hookrightarrow F'$. Note that since $P_\infty$ has degree $1$, it is totally inert in $F'$. In particular, there is a single prime of $F'$ that lies above $P_\infty$, which we identify with $P_\infty$. Given any prime $v$ of $F$, set $d_v$ to denote $-\op{ord}_v$ and for any function $g\in F$, we refer to $d_v(g)$ as the order of the pole of $g$ at $v$. Given a prime $w$ of $F'$ (i.e., point $w\in \mathfrak{X}(\bar{\kappa})$) and $g\in F'$, denote by $d_w(g)$ the order of the pole of $g$ at $w$. We set $d:A\rightarrow \Z_{\geq 0}$ to denote $d_{P_\infty}$.
\begin{lemma}\label{lemma 1}
Let $f,g\in A$ be non-zero. The following assertions hold.
\begin{enumerate}
    \item\label{c1} $d(g)=0$ if and only if $g$ is a constant function.
    \item\label{c2} We have that $d(fg)=d(f)+d(g)$.
    \item\label{c3} Suppose that $d(f)< d(g)$. Then, $d(g+f)=d(g)$.
    \item \label{c4}We have that $d(f)\geq 0$, and $d(f)>0$ if and only if $f$ is non-constant. 
\end{enumerate}
\end{lemma}

\begin{proof}
The proof of parts \eqref{c1} to \eqref{c3} are easy, hence, omitted. For part \eqref{c4} note that a non-constant function $f\in A$ must have a pole at some point. By virtue of being contained in $A$, $f$ does not have poles outside $\{P_\infty\}$. Therefore, $f$ must have a pole at $P_\infty$, and thus, $d(f)>0$. On the other hand, if $f$ is constant, then $d(f)=0$. This proves part \eqref{c4}.
\end{proof}
\begin{lemma}\label{lemma 2}
Let $Y\in A$ and $c_1, c_2\in \bar{\kappa}$ be non-zero constants. If for some prime $p\neq \ell$ we have that 
\[(Y+c_1)^p-Y^p=c_2,\]
then $Y$ is a constant.
\end{lemma}
\begin{proof}
Setting $f(z):=(z+c_1)^p-z^p-c_2$, we find that $f(z)$ is a nonzero polynomial in $z$ with coefficients in $\bar{\kappa}$. Therefore, any solution $Y$ to the equation $f(Y)=0$ must also lie in $\bar{\kappa}$.
\end{proof}

\begin{proof}[Proof of Theorem \ref{main}]
 First consider the case when $p=q$. Note that it is assumed that $p\neq \ell$. We show that there are no non-constant solutions to 
\[X^p-Y^p=1\] in $A$.
Note that $(X-Y)$ divides $X^p-Y^p=1$, hence by Lemma \ref{lemma 1}, \[d(X-Y)=d(1)-d\left(X^{p-1}+X^{p-2}Y+\dots+XY^{p-2}+Y^{p-1}\right)\leq d(1)=0.\] It follows from Lemma \ref{lemma 1} part \eqref{c1} that $(X-Y)$ is a constant $c\in \bar{\kappa}$. We thus deduced that
\begin{equation}\label{eq 2}(Y+c)^p-Y^p=1.\end{equation}

Lemma \ref{lemma 2} implies that \eqref{eq 2} has no non-constant solutions. Since $Y$ is a constant, it follows that $X$ is as well. If $X$ and $Y$ are in $\cO_F$, it follows therefore that $X,Y\in \kappa$.

\par We assume therefore that $p$ and $q$ are distinct (and distinct from $\ell$). Note that there are further conditions on $p$ and $q$. First, we consider the case when $q\nmid h_{F(\mu_p)}$. All the variables introduced in the following argument will be contained in $F(\mu_p)$. Let $\zeta$ be a primitive $p$-th root of $1$. Since it is assumed that $p\neq \ell$, we note that $\zeta\neq 1$. In what follows we consider divisors over $F(\mu_p)$. Given a divisor $D=\sum_v n_v v$ involving primes $v$ of $F(\mu_p)$, the support consists of all primes $v$ such that the coefficient $n_v$ is not equal to $0$. Factor $X^p-1$ into linear factors to obtain the following equation
\begin{equation}\label{eq 3}
    Y^q=\prod_{j=0}^{p-1} (X-\zeta^j).
\end{equation}
For $i\neq j$, note that $(X-\zeta^i)-(X-\zeta^j)=\zeta^j-\zeta^i$, which is a non-zero element of $\kappa(\mu_p)$. Hence, it follows that $\op{div}(X-\zeta^i)$ and $\op{div}(X-\zeta^j)$ have disjoint supports for $i\neq j$. From \eqref{eq 3}, we have the following relation between divisors that are formal linear combinations of primes in $F(\mu_p)$
\[\sum_{j=0}^{p-1}\op{div}(X-\zeta^j)=q\op{div}(Y).\] The elements $(X-\zeta^j)$ are all contained in $F(\mu_p)$, while $Y$ is contained in $F$. Since the divisors $\op{div}(X-\zeta^j)$ have disjoint supports for $i\neq j$, it follows that for each $i$, there is a divisor $D_i$ (involving linear combinations of primes in $F(\mu_p)$) such that $\op{div}(X-\zeta^i)=q D_i$. Since $\op{div}(X-\zeta^i)$ is a principal divisor, it has degree $0$, and hence $D_i$ does also have degree zero. Since $q\nmid h_{F(\mu_p)}$, there is no non-trivial $q$ torsion in the divisor class group. As a result, $D_i$ is a principal divisor $\op{div}(\alpha_i)$, where $\alpha_i\in F(\mu_p)$. Thus, we have deduced that for all $i$, \[X-\zeta^i=u_i\alpha_i^q,\] where $u_i\in F(\mu_p)$ is a non-zero function for which $\op{div}(u_i)=0$. Therefore $u_i$ is a unit, and consequently, is contained in $\kappa(\mu_p)$. Recall that $p$ and $q$ are distinct, and we have shown that $u_i\in \bar{\kappa}$. It follows that $u_i$ is the $q$-th power of an element $v_i\in\bar{\kappa}^\times$. Replacing $\alpha_i$ with $v_i\alpha_i$, we write \[X-\zeta^i=\alpha_i^q,\]where $\alpha_i\in (F')^\times$. Note that $\alpha_i$ is contained in $A$ since it has no poles outside $\{P_\infty\}$ (since $X-\zeta^i$ does not). We deduce that
\begin{equation}\label{eq 5}\alpha_0^q-\alpha_1^q=(X-1)-(X-\zeta)=\zeta-1.\end{equation}It follows that $\alpha_0-\alpha_1$ divides $\zeta-1$, hence has no zeros or poles. As a result, $\alpha_0-\alpha_1$ is a constant $c\in \bar{\kappa}$. It is clear from \eqref{eq 5} that $c$ is non-zero. Thus we find that 
\[(\alpha_1+c)^q-\alpha_1^q=\zeta-1.\]
Lemma \ref{lemma 2} then implies that $\alpha_1$ and $\alpha_0$ are constants. We have thus shown that $X$, and hence $Y$ are both elements in $\bar{\kappa}$. Since $\kappa$ is the algebraic closure of $\F_\ell$ in $F$, and both $X$ and $Y$ are contained in $F$, it follows that $X,Y\in \kappa$.
\par It follows from the condition \eqref{c2 of main} of Theorem \ref{main} that if $p\neq q$, then $q\nmid h_{F(\mu_p)}$ or $p\nmid h_{F(\mu_q)}$. We have shown that there are no non-constant solutions when $p=q$, or when $q\nmid h_{F(\mu_p)}$. Throughout the rest of this proof, we shall therefore assume that $p\nmid h_{F(\mu_q)}$. If both $p$ and $q$ are odd, then we may replace $X$ with $-Y$ and $Y$ with $-X$ to obtain the equation $X^q-Y^p=1$, and thus the previous argument that gives the result applies in this case. We have therefore dealt with the case when both $p$ and $q$ are odd, and we are left to consider the case when $p\nmid h_{F(\mu_q)}$ and either $p$ or $q$ is $2$.

\par First consider the case when $p=2$. It has been shown that there no non-constant solutions when $p=q$ and therefore $q$ must be odd. Moreover, as stated in the previous paragraph, we assume that $2\nmid h_{F(\mu_q)}$. Then, we find that $X^2=Y^q+1=Y^q-(-1)^q=\prod_j (Y+\zeta^j)$, where $\zeta$ is a primitive $q$-th root of unity. For $i\neq j$, note that $(Y+\zeta^i)-(Y+\zeta^j)=\zeta^i-\zeta^j$ is a constant, and therefore, $\op{div}(Y+\zeta^i)$ and $\op{div}(Y+\zeta^j)$ have disjoint supports for $i\neq j$. We thus arrive at the equation
\[\sum_{j=0}^{q-1}\op{div}(Y+\zeta^j)=2\op{div}(X).\] The divisors $\op{div}(Y+\zeta^j)$ have disjoint supports for $i\neq j$, and therefore, we may write $\op{div}(Y+\zeta^j)=2 D_j$ for some divisors $D_j$ that are defined over $F(\mu_q)$. Recall that $2\nmid h_{F(\mu_q)}$. An identical argument to the previous case implies that for all $j$, we have that\[Y+\zeta^j=\beta_j^2,\]where $\beta_j\in A$. We deduce that
\[\beta_0^2-\beta_1^2=(Y+1)-(Y+\zeta)=1-\zeta.\]It follows that $\beta_0-\beta_1$ divides $1-\zeta$. Therefore, $\beta_0-\beta_1$ has no zeros or poles, and hence equals a constant $c\in \bar{\kappa}$. Thus we find that 
\[(\beta_1+c)^2-\beta_1^2=1-\zeta.\]
Lemma \ref{lemma 2} then implies that $\beta_1=Y+\zeta$ and $\beta_0=Y+1$ are constants. From this, we deduce that both $X$ and $Y$ are constants.

\par Finally, assume that $p$ is odd, $q=2$. Note that the result has been proved when $q\nmid h_{F(\mu_p)}$. Therefore, we assume that $q\mid h_{F(\mu_p)}$. It follows from the condition \eqref{c3 of main} of Theorem \ref{main} that $p\nmid h_{F(\mu_4)}$. We consider the equation $X^p=Y^2+1=(Y+\eta)(Y-\eta)$, where $\eta^2=-1$. Note that $F(\mu_4)=F(\eta)$. Since $p$ does not divide the class number of $F(\eta)$, we find that $Y+\eta=\alpha_0^p$ and $Y-\eta=\alpha_1^p$, where $\alpha_0, \alpha_1$ are elements in $A$. Therefore, $2\eta=\alpha_0^p-\alpha_1^p$. In particular, this implies that $(\alpha_0-\alpha_1)$ is a constant $c$. Since $\eta\neq 0$, it follows that $c\neq 0$. We have the following equation
\[(\alpha_1+c)^p-\alpha_1^p=2\eta.\] The result follows from Lemma \ref{lemma 2}.
\end{proof}

\begin{remark}\label{remark}

At this point, it is pertinent to make a few remarks.
\begin{itemize}
    \item The assumptions that $p$ and $q$ are not equal to $\ell$ are necessary. Indeed, suppose that $p=\ell$. Then, setting $X=1+z^q$ and $Y=z^p$ for any element $z\in \cO_F$, one would obtain non-constant solutions.
    \item The methods introduced in this paper could potentially be applied to a more general class of diophantine equations, namely, equations of the form $X^m=f(Y)$, where $f(Y)\in \kappa[Y]$, where $\kappa$ is the field of constants of $F$. 
\end{itemize}
\end{remark}
\bibliographystyle{unsrt}
\bibliography{references}
\end{document}